\documentclass[a4paper]{amsart}
\pdfoutput=1
\usepackage[utf8]{inputenc}
\usepackage[T1]{fontenc}
\usepackage{lmodern}
\usepackage{amssymb}
\usepackage{mathtools,enumitem}
\usepackage{microtype}

\setlist[enumerate,1]{label=\textup{(\arabic*)}}
\setlist[enumerate,2]{label=\textup{(\alph*)}}

\usepackage[pdftitle={Isomorphism and stable isomorphism in “real” and “quaternionic” K-theory},
pdfauthor={Malkhaz Bakuradze, Ralf Meyer},
pdfsubject={Mathematics}
]{hyperref}
\usepackage[lite]{amsrefs}

\renewcommand{\PrintDOI}[1]{\href{http://dx.doi.org/\detokenize{#1}}{doi: \detokenize{#1}}}

\BibSpec{book}{%
  +{}  {\PrintPrimary}                {transition}
  +{,} { \textit}                     {title}
  +{.} { }                            {part}
  +{:} { \textit}                     {subtitle}
  +{,} { \PrintEdition}               {edition}
  +{}  { \PrintEditorsB}              {editor}
  +{,} { \PrintTranslatorsC}          {translator}
  +{,} { \PrintContributions}         {contribution}
  +{,} { }                            {series}
  +{,} { \voltext}                    {volume}
  +{,} { }                            {publisher}
  +{,} { }                            {organization}
  +{,} { }                            {address}
  +{,} { \PrintDateB}                 {date}
  +{,} { }                            {status}
  +{}  { \parenthesize}               {language}
  +{}  { \PrintTranslation}           {translation}
  +{;} { \PrintReprint}               {reprint}
  +{.} { }                            {note}
  +{.} {}                             {transition}
  +{} { \PrintDOI}                   {doi}
  +{} { available at \url}            {eprint}
  +{}  {\SentenceSpace \PrintReviews} {review}
}

\theoremstyle{plain}
\newtheorem{theorem}[subsection]{Theorem}
\newtheorem{lemma}[subsection]{Lemma}
\newtheorem{proposition}[subsection]{Proposition}

\theoremstyle{definition}
\newtheorem{definition}[subsection]{Definition}

\theoremstyle{remark}
\newtheorem{remark}[subsection]{Remark}

\newcommand*{\defeq}{\mathrel{\vcentcolon=}}

\newcommand{\C}{\mathbb{C}}
\newcommand{\N}{\mathbb{N}}
\newcommand{\Z}{\mathbb{Z}}
\newcommand{\R}{\mathbb{R}}
\newcommand{\T}{\mathbb{T}}
\newcommand{\Sphere}{\mathbb{S}}
\newcommand{\field}{\mathbb{F}}

\newcommand*{\nb}{\nobreakdash}
\newcommand*{\Cst}{\mathrm C^*}

\newcommand{\Mat}{\mathbb M}

\newcommand*{\conj}[1]{\overline{#1}}

\newcommand*{\Id}{\mathrm{id}}
\newcommand*{\K}{\mathrm{K}}

\DeclareMathOperator{\rank}{rank}

\begin{document}
\title[Nonstable ``real'' and ``quaternionic'' K-theory]{Isomorphism and stable isomorphism in ``real'' and ``quaternionic'' K-theory}
\author{Malkhaz Bakuradze}
\email{malkhaz.bakuradze@tsu.ge}
\address{Faculty of Exact and Natural Sciences, Tbilisi State
  University\\
  Tbilisi\\
  Georgia}

\author{Ralf Meyer}
\email{rmeyer2@uni-goettingen.de}
\address{Mathematisches Institut\\
  Universit\"at G\"ottingen\\Bunsenstra\ss e 3--5\\
  37073 G\"ottingen\\Germany}


\thanks{This work was supported by the Shota Rustaveli National Science Foundation of Georgia (SRNSFG) grant FR-23-779.}

\begin{abstract}
  We find lower bounds on the rank of a ``real'' vector bundle over an
  involutive space, such that ``real'' vector bundles of higher rank
  have a trivial summand and such that a stable isomorphism for such
  bundles implies ordinary isomorphism.  We prove similar lower bounds
  also for ``quaternionic'' bundles.  These estimates have
  consequences for the classification of topological insulators with
  time-reversal symmetry.
\end{abstract}
\maketitle

\section{Introduction}
\label{sec:intro}

Topological insulators are materials that are insulators such that
some special topology enforces the existence of conducting states on
their boundaries.  These conducting boundary states tend to be very
robust under disorder.  See, for instance,
\cite{Schulz-Baldes:Insulators} for a survey on this subject from
the perspective of non-commutative geometry and index theory.  It is
common to model such materials in the one-particle approximation.
The physical system is then described through a
\(\Cst\)\nb-algebra~\(A\) of observables and an invertible,
self-adjoint element \(H\in A\), the Hamiltonian.  Two such systems
with the same \(\Cst\)\nb-algebra~\(A\) are considered in the same
topological phase if there is a homotopy of invertible self-adjoint
elements in~\(A\) between their Hamiltonians.  Up to homotopy, we
may replace self-adjoint invertible elements by self-adjoint
unitaries.  By a linear transformation, these may be replaced by
projections in~\(A\).  So the set of possible topological phases of
the system described by~\(A\) is the set of
homotopy classes of projections in~\(A\).

A projection in~\(A\) has a \(\K\)\nb-theory class in \(\K_0(A)\).
However, if two projections have the same class in \(\K_0(A)\), then
they are only stably homotopic, that is, they become homotopic after
adding the same projection to them both.  Since \(\K\)\nb-theory is
easier to compute than sets of homotopy classes of projections, it is
important to know situations where stable homotopy implies homotopy.

If disorder in the system is neglegcted, then the observable algebra
becomes isomorphic to a matrix algebra over the algebra of continuous
functions on the \(d\)\nb-torus~\(\T^d\), where~\(d\) is the dimension
of the material.  Many interesting topological materials only exhibit
a nontrivial topological phase when we require extra symmetries that
are anti-unitary or that anticommute with the Hamiltonian.  A
particularly important case is a time-reversal symmetry.  This is
actually two cases because the symmetry may have square \(+1\)
or~\(-1\).  It is explained in \cites{De_Nittis-Gomi:Real_Bloch,
  De_Nittis-Gomi:Quaternionic_Bloch} how topological phases in these
two symmetry types are classified using ``real'' and ``quaternionic''
vector bundles over tori with involution.  We briefly sketch this here.

The torus~\(\T^d\) above appears through the Fourier transform.  The
relevant observable algebra is the group
\(\Cst\)\nb-algebra of~\(\Z^d\), tensored by a matrix algebra.  Here
the time-reversal symmetry acts by entrywise complex conjugation
combined with conjugation by a suitable scalar matrix~\(\Theta\).  Under
Fourier transform, this becomes the real \(\Cst\)\nb-algebra
\[
  \bigl\{f\colon \T^d \to \Mat_n(\C) \bigm|
  f(\conj{z}) = \Theta \conj{f(z)} \Theta^{-1} \text{ for all }z\in\T^d
  \bigr\}.
\]
The conjugation map \(z\mapsto \conj{z}\) on~\(\T^d\) is an involution
and generates an action of the group~\(\Z/2\).  Projections in matrix
algebras over the real group \(\Cst\)\nb-algebra of~\(\Z^d\)
correspond by the Serre--Swan Theorem to ``real'' vector bundles
over~\(\T^d\) as defined by Atiyah~\cite{Atiyah:K_Reality}.  That is,
they carry a map on their total space that lifts the involution on the
base space, is fibrewise conjugate-linear, and squares to the identity map.  For
a time reversal symmetry with square~\(-1\), we instead need the map
on the total space to square to the map of multiplication by~\(-1\),
and complex bundles with this kind of extra structure are called
``quaternionic''.

Thus it is physically interesting to know sufficient criteria for
two ``real'' or ``quaternionic'' vector bundles over~\(\T^d\) that
are stably isomorphic to be isomorphic.  Without the extra ``real''
or ``quaternionic'' structure map, such criteria are well known: for
each~\(d\) there is an explicit \(k(d)\in\N\) such that two complex
vector bundles of rank at least~\(k(d)\) over a space of covering
dimension~\(d\) are isomorphic once they are stably isomorphic (see
\cite{Husemoller:Fibre_bundles_3}*{Theorem~1.5 in Chapter~9}).  Similar
results are known for real and quaternionic vector bundles, but the
more general cases of ``real'' and ``quaternionic'' vector bundles
have not yet been treated.  This will be done here.

A related result says that any ``real'' vector bundle of
sufficiently high rank is a direct sum of a trivial ``real'' vector
bundle of rank~\(1\) and another ``real'' vector bundle.  In fact, a
relative version of this result, saying that a trivial direct
summand on a subspace may be extended to one on the whole space,
implies that stable isomorphism and isomorphism are equivalent for
bundles of sufficiently high rank.  However, to prove that stable
isomorphism and isomorphism are equivalent for certain bundles over
a space~\(X\), we need the statement for the
space \(X\times [0,1]\), relative to the subspace \(X\times \{0,1\}\).

Some results about trivial direct summands in vector bundles of
sufficiently high rank are already proven in
\cite{De_Nittis-Gomi:Real_Bloch}*{Theorem~4.25}
for ``real'' vector bundles and
\cite{De_Nittis-Gomi:Quaternionic_Bloch}*{Theorem~2.5}
for ``quaternionic'' vector bundles.  However, these statements assume
that the set of fixed points of the ``real'' involution on the
underlying space is discrete, so that they never apply to a space of
the form \(X\times [0,1]\).  Thus they cannot help to relate stable
isomorphism to isomorphism.  Therefore, our main task is to remove
this restriction from the results in \cites{De_Nittis-Gomi:Real_Bloch,
  De_Nittis-Gomi:Quaternionic_Bloch}.

As in \cites{De_Nittis-Gomi:Real_Bloch,
  De_Nittis-Gomi:Quaternionic_Bloch}, we work with
\(\Z/2\)-CW-complexes.  The key proof technique is induction over
cells, extending certain equivariant maps already defined on the
boundary of a $\Z/2$-cell to the interior.  Any smooth manifold with a
smooth \(\Z/2\)\nb-action carries an equivariant triangulation
by~\cite{Illman:Equivariant_triangulations}, and this implies
immediately that it may be turned into a \(\Z/2\)-CW-complex.  We now
formulate our main results.

\begin{theorem}
  \label{the:main_real}
  Let \(d_1,d_0,k\in\N\).  Let~\(X\) be a \(\Z/2\)-CW-complex.  Assume
  that the free cells in~\((X,A)\) have at most dimension~\(d_1\) and
  that the trivial cells have at most dimension~\(d_0\).  Let
  \[
    k_0 \defeq
    \max \left\{ d_0,
      \left\lceil \frac{d_1-1}{2} \right\rceil \right\},
    \qquad
    k_1 \defeq
    \max \left\{ d_0+1,
       \left\lceil \frac{d_1}{2} \right\rceil \right\}.
  \]
  \begin{enumerate}
  \item \label{the:main_real_1}%
    Let~\(E\) be a ``real'' vector bundle over~\(X\) of rank
    \(k\ge k_0\).  There is an isomorphism
    \(E\cong E_0 \oplus (X\times \C^{k-k_0})\) for some ``real''
    vector bundle~\(E_0\) over~$X$ and the trivial ``real'' vector
    bundle $X\times \C^{k-k_0}$ of rank~$k-k_0$.
  \item \label{the:main_real_2}%
    Let \(E_1\) and~\(E_2\) be two ``real'' vector bundles over~\(X\)
    of rank \(k\ge k_1\).  If \(E_1\) and~\(E_2\) are stably
    isomorphic, that is, $E_1 \oplus E_3 \cong E_2 \oplus E_3$ for
    some ``real'' vector bundle~$E_3$, then they are isomorphic.
  \end{enumerate}
\end{theorem}

\begin{theorem}
  \label{the:main_quaternion}
  Let \(d_1,d_0,k\in\N\).  Let~\(X\) be a \(\Z/2\)-CW-complex.  Assume
  that the free cells in~\((X,A)\) have at most dimension~\(d_1\) and
  that the trivial cells have at most dimension~\(d_0\).  Let
  \[
    k_0 \defeq \max \left\{ \left\lceil\frac{d_0-3}{2}\right\rceil,
      \left\lceil \frac{d_1-1}{2} \right\rceil \right\},\qquad
    k_1 \defeq \max \left\{ \left\lceil\frac{d_0-2}{2} \right\rceil,
    \left\lceil \frac{d_1}{2} \right\rceil \right\}.
  \]
  \begin{enumerate}
  \item \label{the:main_quaternion_1}%
    Let~\(E\) be a ``quaternionic'' vector bundle over~\(X\) of rank
    \(k\ge k_0\).  There is an isomorphism
    \(E\cong E_0 \oplus \theta^{2\lfloor (k-k_0)/2\rfloor}_X\) for
    some ``quaternionic'' vector bundle~\(E_0\) over~$X$ and the
    trivial ``quaternionic'' vector
    bundle~$\theta^{2\lfloor (k-k_0)/2\rfloor}_X$ of
    rank~$2\lfloor (k-k_0)/2\rfloor$.
  \item \label{the:main_quaternion_2}%
    Let \(E_1\) and~\(E_2\) be two ``quaternionic'' vector bundles
    over~\(X\) of rank \(k\ge k_1\).  If \(E_1\) and~\(E_2\) are
    stably isomorphic, that is, $E_1 \oplus E_3 \cong E_2 \oplus E_3$
    for some ``quaternionic'' vector bundle~$E_3$, then they are
    isomorphic.
  \end{enumerate}
\end{theorem}

\begin{proposition}
  \label{pro:torus_dim_four}
  Two ``real'' or ``quaternionic'' vector bundles over~\(\T^d\) with
  \(d\le 4\) are already isomorphic once they are stably isomorphic.
\end{proposition}

\begin{proof}
  Here we are dealing with the \(\Z/2\)-CW-complex~\(\T^d\), which has
  $d_0=0$ and $d_1 =d\le 4$.  This gives $k_1 \le 2$ both in
  Theorem~\ref{the:main_real} and in
  Theorem~\ref{the:main_quaternion}.  So the assertion holds for
  ``real'' and ``quaternionic'' bundles of rank at least~\(2\).  All
  rank-zero bundles are trivial.  So it only remains to prove the
  statement for bundles of rank one.  All rank-one ``real'' bundles
  over~\(\T^d\) for all \(d\in\N\) are trivial by
  \cite{De_Nittis-Gomi:Real_Bloch}*{Proposition~5.3}.  Since the torus
  has $\tau$\nb-fixed points, any ``quaternionic'' bundle
  over~\(\T^d\) has even rank
  (see~\cite{De_Nittis-Gomi:Quaternionic_Bloch}), so there are no
  ``quaternionic'' vector bundles of rank one.
\end{proof}

\begin{remark}
  Consider the stabilisation map $[E]\mapsto [E\oplus \theta^{m-k}_X]$
  between the sets of ``real'' or ``quaternionic'' vector bundles of
  rank $k$ and~$m$ for $m\ge k$; in the ``quaternionic'' case, this
  only works if $m-k$ is even.  Theorems \ref{the:main_real}
  and~\ref{the:main_quaternion} say that this map is injective for
  $k\ge k_1$ and surjective for $k\ge k_0$.  So it is bijective for
  $k\ge \max\{k_0,k_1\} = k_1$.  For instance, if $d_0,d_1\le 4$,
  this happens if $k\ge2$.  So in these low dimensions, it is no loss
  of generality to restrict attention to ``real'' and ``complex''
  vector bundles of rank at most~$2$.  This is claimed in
  \cite{De_Nittis-Gomi:Quaternionic_Bloch}*{Corollary~2.1} for $d_0=0$
  and $d_1=5$ as well.  We can only confirm the surjectivity of the
  map in this case, however, and the injectivity of the map is not
  addressed in~\cite{De_Nittis-Gomi:Quaternionic_Bloch}.  If $d_0=0$,
  then the threshold~$k_0$ for the map above to be surjective is the
  same as the threshold $\lfloor d/2\rfloor$ in the ``real'' case in
  \cite{De_Nittis-Gomi:Quaternionic_Bloch}*{Theorem~4.25} and the
  threshold $2\lfloor (d+2)/4\rfloor$ in the ``quaternionic'' even
  rank case in \cite{De_Nittis-Gomi:Quaternionic_Bloch}*{Theorem~2.5}.
\end{remark}

In the body of the paper, we state and prove generalisations of
Theorems \ref{the:main_real} and~\ref{the:main_quaternion} for
relative $\Z/2$-CW-complexes, which allow to extend a given direct sum
decomposition on a subspace.  We need the relative versions of the
first statements in Theorems \ref{the:main_real}
and~\ref{the:main_quaternion} to prove the second statements.
Section~\ref{sec:real} contains our results and some notation in the
``real'' case, and Seection~\ref{sec:quaternion} treats the
``quaternionic'' case.  Finally, in Section~\ref{sec:conjugacy}, we
explain how our results imply statements about stable conjugacy and
conjugacy of projections in the physical observable $\Cst$\nb-algebra.

An \emph{involutive space} $(X,\tau)$ is a topological space~$X$ with
a continuous involution $\tau\colon X\to X$, that is,
$\tau^2 = \Id_X$.  Throughout this article, let $(X,A)$ be a relative
$\Z/2$-CW-complex, that is, $A\subseteq X$ is a closed
$\tau$\nb-invariant subspace and~$X$ is gotten from~$A$ by attaching
$\Z/2$-cells of increasing dimensions.  There are two different types
of $\Z/2$-cells, namely, the \emph{free cells}
$\mathbb{D}^j \times \Z/2$ with the generator of~$\Z/2$ acting by
$\tau(x,j) \defeq (x,j+1)$ and the \emph{fixed cells} $\mathbb{D}^j$
with the trivial $\Z/2$-action.  Let~$d_0$ be the supremum of the
dimensions of the fixed cells, which is the dimension of
$X^\tau\setminus A$.  Let~$d_1$ be the supremum of the dimensions of
the free cells, which is the dimension
of~$X\setminus (X^\tau \cup A)$.  Our results only work if
$d_0,d_1<\infty$.

\section{``Real'' vector bundles}
\label{sec:real}

This section proves our main result for ``real'' vector bundles.  Even more,
we state and prove a relative version over relative
$\Z/2$-CW-complexes.

\begin{definition}[\cites{Atiyah:K_Reality, De_Nittis-Gomi:Real_Bloch}]
  A \emph{``real'' vector bundle}
  over an involutive space $(X,\tau)$ is a complex vector bundle
  $\pi \colon E \to X$ with a homeomorphism
  $\Theta\colon E \to E$ such that
  \begin{enumerate}
  \item $\pi \circ \Theta=  \tau  \circ \pi$;
  \item $\Theta$ is fibrewise additive and
    $\Theta(\lambda p)=\conj{\lambda}p$ for all $\lambda \in \C$ and
    $p\in E$, where $\conj{\lambda}$ is the complex conjugate of
    $\lambda$;
  \item $\Theta^2=\Id_E$.
  \end{enumerate}
  The bundle has \emph{rank}~$k$ if all its fibres are isomorphic
  to~$\C^k$.
  
  The \emph{trivial ``real'' vector bundle} of rank~$k$ over~$X$ is
  $X\times\C^k$ with the obvious, trivial $\C$\nb-vector bundle
  structure and $\Theta(x,v) \defeq (\tau(x),\conj{v})$.  It is
  denoted by~$\theta^k_X$.
\end{definition}

\begin{proposition}[\cite{De_Nittis-Gomi:Real_Bloch}]
  \label{complexification}
  Let~$X$ be a space and let $k \in \N$.  Then any ``real'' vector
  bundle over $(X,\Id_X)$ is the complexification of an ordinary real
  vector bundle.
\end{proposition}

Therefore, for any involutive space~$(X,\tau)$, the restriction~$E$ on
the subset $X^\tau\subseteq X$ of $\tau$\nb-fixed points is a
complexification of a real vector bundle, namely,
\[
  E|_{X^\tau} \cong E^\Theta\otimes_\R \C,
\]
where~$E^\Theta$ is the set of fixed points of~$\Theta$, which is an
$\R$\nb-vector bundle over~$X^\tau$.

\bigskip

Let~$\field$ denote $\R$, $\C$, or $\mathbb{H}$, and let
$c=\dim_\R \field$.  Recall a classical result:

\begin{proposition}[\cite{Husemoller:Fibre_bundles_3}*{Chapter~9,
    Proposition~1.1}]
  \label{Husemoller}
  Let~$\xi^k$ be a $k$\nb-dimensional $\field$-vector bundle over a
  CW-complex~$X$ with $d \leq ck-1$.  Then~$\xi$ is isomorphic to
  $\eta \oplus (X\times \field)$ for some $\field$\nb-vector
  bundle~$\eta$ over~$X$.
\end{proposition}

The key ingredient in the proof is
\cite{Husemoller:Fibre_bundles_3}*{Theorem~7.1 in Chapter~2}, which
allows to extend sections of fibre bundles under a higher connectedness
assumption.  This proof technique provides a relative
version of the proposition for a relative CW-complex~$(X,A)$, which
shows that a given direct sum decomposition on~$A$ extends to~$X$.

We now formulate and prove a relative version of
Theorem~\ref{the:main_real}.\ref{the:main_real_1}, which generalises
the relative version of Proposition~\ref{Husemoller} for $\field = \R$
to ``real'' bundles.  Our proof follows the proof of
\cite{Husemoller:Fibre_bundles_3}*{Theorem~1.2} and
\cite{De_Nittis-Gomi:Real_Bloch}*{Proposition~4.23}.  Our main task is
to remove the extra assumption in the latter result that fixed point
cells are only of dimension~$0$.

\begin{theorem}
  \label{the:main_real_relative_1}
  Let \(d_1,d_0,k\in\N\).  Let \((X,A)\) be a relative
  \(\Z/2\)-CW-complex.  Assume that the free cells in~\((X,A)\) have
  at most dimension~\(d_1\) and that the trivial cells have at most
  dimension~\(d_0\).  Let
  \[
    k_0 \defeq
    \max \left\{ \left\lceil \frac{d_1-1}{2} \right\rceil,
      d_0  \right\}.
  \]
  Let~\(E\) be a ``real'' vector bundle over~\(X\) of rank
  \(k\ge k_0\).  Let an isomorphism
  $E|_A \cong E_0^A \oplus \theta_A^{k-k_0}$ for some ``real'' vector
  bundle~\(E_0^A\) over~$A$ be given.  This extends to an isomorphism
  \(E\cong E_0 \oplus \theta_X^{k-k_0}\) for some ``real'' vector
  bundle~\(E_0\) over~$X$.
\end{theorem}

\begin{proof}
  We are going to extend an isomorphism
  $E|_A \cong E_0^A \oplus \theta^1_A$ to an isomorphism
  $E\cong E_0\oplus \theta^1_X$, assuming $k> k_0$.  Repeating this
  step $k-k_0$~times then gives the result that is stated.  An
  isomorphism $E\cong E_0\oplus \theta^1_X$ contains an injective
  ``real'' vector bundle map $\theta^1_X \hookrightarrow E$.
  Conversely, such an embedding implies an isomorphism
  $E\cong E_0\oplus \theta^1_X$ because any ``real'' vector subbundle
  has an orthogonal complement, which is again a ``real'' vector
  subbundle, and then the direct sum is isomorphic to the whole bundle
  (see~\cite{De_Nittis-Gomi:Real_Bloch}).  An injective ``real''
  vector bundle map $\theta^1_X \hookrightarrow E$ is equivalent to a
  section $s\colon X\to E$ that satisfies $s(x)\neq0$ and
  $\Theta(s(x)) = s(\tau(x))$ for all $x\in X$: then we map
  $\theta^1_X=X\times\C$ to~$E$ by
  $(x,\lambda)\mapsto \lambda\cdot s(x)$.  We call the section
  $\Z/2$-equivariant if $\Theta(s(x)) = s(\tau(x))$ for all $x\in X$.
  Let $E^\times \subset E$ be the subbundle of nonzero vectors.  Our
  task is to extend a $\Z/2$-equivariant section $A\to E^\times|_A$ to
  a $\Z/2$-equivariant section $X\to E^\times$.  The fibres
  of~$E^\times$ are $(\C^k)^\times=\C^k\setminus\{0\}$.  This is
  homotopy equivalent to the sphere~$\Sphere^{2k-1}$, which is
  $2k-2$-connected.

  First, we construct our section on $A\cup X^\tau$.  This is
  equivalent to extending a given section on~$A^\tau$ to a section
  $X^\tau\to (E^\Theta)^\times$ (compare
  Proposition~\ref{complexification}).  The $\R$\nb-vector bundle
  $E^\Theta \twoheadrightarrow X^\tau$ has dimension~$k$.  The proof of
  Proposition~\ref{Husemoller} also allows to extend a section that is
  given on a closed subspace.  The assumption $d_0 \le k_0 \le k-1$
  ensures that the section that is given on~$A^\tau$ may be extended
  to a section $X^\tau\to (E^\Theta)^\times$.  Together with the given
  section on~$A$, we get the desired $\Z/2$-equivariant section
  of~$E^\times$ on $A\cup X^\tau$.

  We prolong this section to all of~$X$ by induction over skeleta.
  Since the cells of the same dimension are disjoint, we may work one
  $\Z/2$-cell at a time.  Since we have already found the section
  on~$X^\tau$, we only encounter free cells of the form
  $\Z/2 \times \mathbb{D}^j$, and $j\le d_1$ by assumption.  We
  are given a $\Z/2$-equivariant section on the boundary
  $\Z/2 \times \partial\mathbb{D}^j$, which we have to extend to
  $\Z/2 \times \mathbb{D}^j$.  The involution~$\tau$ flips the two
  copies of~$\mathbb{D}^j$ in $\Z/2 \times \partial\mathbb{D}^j$.  So
  it suffices to construct the section
  $s\colon \{+1\}\times \mathbb{D}^j \to E^\times$ and then define
  $s(-1,x) \defeq \Theta(s(1,x))$.  This is automatically
  $\Z/2$-equivariant.  The restriction of the bundle~$E$ to
  $\{+1\}\times \mathbb{D}^j$ is trivial because~$\mathbb{D}^j$ is
  contractible.  So our task becomes equivalent to extending a map
  $\partial \mathbb{D}^j \to (\C^k)^\times$ to a map
  $\mathbb{D}^j \to (\C^k)^\times$.  This is possible because
  $j\le d_1\le 2k_0 +1\le 2k-1$.
\end{proof}

\begin{theorem}
  \label{the:main_real_relative_2}
  Let \(d_1,d_0,k\in\N\).  Let \((X,A)\) be a relative
  \(\Z/2\)-CW-complex.  Assume that the free cells in~\((X,A)\) have
  at most dimension~\(d_1\) and that the trivial cells have at most
  dimension~\(d_0\).  Let
  \[
    k_1 \defeq
    \max \left\{ \left\lceil \frac{d_1}{2} \right\rceil,
      d_0+1  \right\}.
  \]
  Let \(E_1\) and~\(E_2\) be two ``real'' vector bundles over~\(X\) of
  rank \(k\ge k_1\).  An isomorphism $E_1|_A \cong E_2|_A$ that
  extends to a stable isomorphism between \(E_1\) and~\(E_2\) on~$X$
  extends to an isomorphism \(E_1\cong E_2\).
\end{theorem}

\begin{proof}
  Let $\varphi_A\colon E_1|_A \cong E_2|_A$ be the given isomorphism
  on~$A$.  Any ``real'' vector bundle over a finite-dimensional
  $\Z/2$-CW-complex is a direct summand in a trivial ``real'' bundle.
  Therefore, our stable isomorphism assumption implies that there is
  an isomorphism
  $\psi\colon E_1 \oplus \theta^\ell_X \cong E_2 \oplus \theta_X^\ell$
  for some $\ell\ge0$ such that~$\psi|_A$ is
  $\varphi_A \oplus \Id_{\theta_X^\ell}$.  There is nothing to do if
  $\ell=0$.  We are going to prove that there is an isomorphism
  $E_1 \oplus \theta_X^{\ell-1} \cong E_2 \oplus \theta_X^{\ell-1}$
  that extends $\varphi_A \oplus \Id_{\theta_X^{\ell-1}}$.  Repeating
  this step $\ell$~times gives the result we need.
  Replacing~$E_j$ by $E_j \oplus \theta_X^{\ell-1}$, we reduce to the
  case where $\ell=1$.  Thus we may assume an isomorphism
  $\psi\colon E_1 \oplus \theta_X^1 \cong E_2 \oplus \theta_X^1$ in
  the following.

  We want to reduce the proof to
  Theorem~\ref{the:main_real_relative_1}, as in the proof of
  \cite{Husemoller:Fibre_bundles_3}*{Theorem~1.5 in Chapter~9}.  We work
  on $Y=X\times I$ with $I=[0,1]$, equipped with the involution
  $\tau(x,t) \defeq (\tau(x),t)$.  We let~$E$ be the pullback of
  $E_1\oplus \theta_X^1$ to~$Y$.  The relevant dimensions and ranks are
  now
  \[
    d_1(Y)=d_1(X)+1,\qquad
    d_0(Y)=d_0(X)+1,\qquad
    \rank(E)=k+1.
  \]
  We identify the restriction of~$E$ to
  $B\defeq X\times \partial I \cup A\times I \subseteq Y$ with a
  ``real'' vector bundle of the form $E_0 \oplus \theta^1_B$.  Here we
  glue the identity isomorphism from~$E$ to the pull back of
  $E_1 \oplus \theta_X^1$ on $X\times\{0\} \cup A \times [0,1]$ and
  the isomorphism
  $E|_{X\times\{1\}} = E_1 \oplus \theta_X^1 \cong E_2 \oplus
  \theta_X^1$ on $X\times\{1\}$; we may glue this on $A\times\{1\}$
  because the isomorphism~$\psi$ is of the form
  $\varphi_A\oplus \Id_{\theta_A^1}$ on~$A$.  Now
  Theorem~\ref{the:main_real_relative_1} provides an isomorphism
  $E \cong E_0\oplus \theta_X^1$ on all of~$Y$ extending the given
  isomorphism on $B \subseteq Y$.  By construction, the bundle~$E_0$
  on~$Y$ restricts to~$E_1$ on $X\times\{0\} \cup A \times [0,1]$ and
  to~$E_2$ on $X\times\{1\}$, glued together using the given
  isomorphism~$\varphi_A$.  Now, as in the proof
  in~\cite{Husemoller:Fibre_bundles_3}, the existence of such a
  ``real'' vector bundle over~$Y$ implies that there is an isomorphism
  of ``real'' vector bundles $E_1 \cong E_2$ that restricts to the
  given isomorphism on~$A$.
\end{proof}

\section{``Quaternionic'' vector bundles}
\label{sec:quaternion}

The goal of this section is to prove a relative version of
Theorem~\ref{the:main_quaternion}.

\begin{definition}[\cite{De_Nittis-Gomi:Quaternionic_Bloch}]
  A \emph{``quaternionic'' bundle}
  over the involutive space $(X,\tau)$ is a complex vector bundle
  $\pi\colon E \to X$ with a homeomorphism
  $\Theta\colon E\to E$ such that
  \begin{itemize}
  \item $\pi \circ \Theta= \tau \circ \pi$;
  \item $\Theta$ is fibrewise additive and
    $\Theta(\lambda p)=\conj{\lambda}p$ for all $\lambda \in \C$ and
    $p\in E$;
  \item $\Theta^2(x,v) = (x,-v)$ is fibrewise multiplication by~$-1$
    for all $x\in X$, $v\in E_x$.
  \end{itemize}
  The bundle has \emph{rank}~$k$ if all its fibres are isomorphic
  to~$\C^k$.
\end{definition}

The name ``quaternionic'' vector bundles is justified by the
following:

\begin{proposition}[\cite{De_Nittis-Gomi:Quaternionic_Bloch}]
  \label{Q}
  A ``quaternionic'' vector bundle over $(X,\Id_X)$ is equivalent to
  an $\mathbb{H}$\nb-vector bundle, where $a+ b i + c j + d k$ acts by
  $(a+b i) + (c+ d i)\Theta$ on each fibre.  The rank as a
  ``quaternionic'' vector bundle is twice the rank as an
  $\mathbb{H}$\nb-vector bundle because $\mathbb{H}^k= \C^{2 k}$.
\end{proposition}

In particular, the restriction of a ``quaternionic'' vector bundle
to~$X^\tau$ must have even rank.  If~$X$ is connected and
$X^\tau\neq\emptyset$, then this implies that the rank is even on all
of~$X$.  Nevertheless, ``quaternionic'' vector bundles of odd rank are
possible if $X^\tau=\emptyset$.  All trivial ``quaternionic'' bundles
have even rank.  Namely, the trivial ``quaternionic'' vector
bundle~$\theta^{2 k}_X$ over $(X,\tau)$ of rank~$2 k$ is the space
$X \times \C^{2 k}$ with
\[
  \Theta(x,\lambda_1,\lambda_2,\dotsc,\lambda_{2 k -1},\lambda_{2 k})
  \defeq (\tau(x),\conj{\lambda_2},-\conj{\lambda_1},\dotsc,
  \conj{\lambda_{2 k}}, -\conj{\lambda_{2 k -1}}).
\]

A vector bundle map $f\colon \theta_X^1 \to E$ is of the form
$f(x,\lambda_1,\lambda_2) = \lambda_1 s_1(x) + \lambda_2 s_2(x)$ for
two sections $s_1,s_2$ of~$E$.  This map is $\Z/2$-equivariant if and
only if $s_2(x) = -\Theta(s_1(\tau(x)))$ for all $x\in X$.  Thus the
section~$s_1$ already determines~$f$ if it is $\Z/2$-equivariant.  Of
course, $f$ is injective if and only if $s_1(x)$ and $s_2(x)$ are
linearly independent for all $x\in X$.  If $\tau(x)=x$, this is true
once $s_1(x)\neq0$ because then the restriction of~$f$ to the fibre
at~$x$ is an $\mathbb{H}$\nb-linear map $\mathbb{H} \to E_x$.  If,
however, $x\neq \tau(x)$, then $s_1(x)\neq0$ is not sufficient.  We
must ensure that $s_2(x) =-\Theta(s_1(\tau(x)))$ is linearly
independent of $s_1(x)$ as well.  Here a mistake is made
in~\cite{De_Nittis-Gomi:Quaternionic_Bloch}: their argument above
Definition~2.1 why $s_1(x)\neq0$ should suffice for $s_1(x)$
and~$s_2(x)$ to be linearly independent is wrong because it only
implies that the functions $s_1$ and~$s_2$ are linearly independent,
which is much weaker; so their proofs of Propositions 2.4 and~2.7 are
incomplete.  It is easy to fix the proof of their Proposition~2.4, and
our theorem above may replace their Proposition~2.7.


\begin{theorem}
  \label{the:main_quaternion_relative_1}
  Let \(d_1,d_0,k\in\N\).  Let \((X,A)\) be a relative
  \(\Z/2\)-CW-complex.  Assume that the free cells in~\((X,A)\) have
  at most dimension~\(d_1\) and that the trivial cells have at most
  dimension~\(d_0\).  Let
  \[
    k_0 \defeq \max \left\{ \left\lceil\frac{d_0-3}{2}\right\rceil,
      \left\lceil \frac{d_1-1}{2} \right\rceil \right\}.
  \]
  Let~\(E\) be a ``quaternionic'' vector bundle over~\(X\) of rank
  \(k\ge k_0\).  Assume an isomorphism
  $E|_A \cong E_0^A \oplus \theta_A^{2\lfloor (k - k_0)/2\rfloor}$ for
  some ``quaternionic'' vector bundle~\(E_0^A\) over~$A$ is given.
  This isomorphism extends to an isomorphism
  \(E\cong E_0 \oplus \theta_X^{2\lfloor (k - k_0)/2\rfloor}\) for
  some ``quaternionic'' vector bundle~\(E_0\) over~$X$.
\end{theorem}

\begin{proof}
  We are going to prove that any injective $\Z/2$-equivariant vector
  bundle map $f_A\colon \theta^2_A \to E|_A$ extends to an injective
  $\Z/2$-equivariant vector bundle map $f\colon \theta^2_X \to E$ if
  $k\ge k_0 +2$.  This implies the statement as in the proof of
  Theorem~\ref{the:main_real_relative_1}.

  We first construct~$f$ on the subset $X^\tau \cup A$.  This is
  equivalent to extending $f_A|_{A^\tau}$ from~$A^\tau$ to~$X^\tau$.
  In this part of the proof, we may assume without loss of generality
  that $X^\tau\neq\emptyset$.  This forces~$k$ to be even.  Since the
  involution acts trivially on~$X^\tau$, the ``quaternionic''
  bundle~$E$ of rank~$k$ becomes an $\mathbb{H}$\nb-vector bundle of
  rank~$k/2$.  Our assumptions imply that all cells in the relative
  CW-complex $(X^\tau,A^\tau)$ have dimension
  $j \le d_0 \le 2 k_0 + 3 \le 2 k - 1 = 4(k/2) - 1$.  Now the
  relative version of Proposition~\ref{Husemoller} allows us to
  extend~$f_A|_{A^\tau}$ to an injective $\Z/2$-equivariant vector
  bundle map $\theta^1_{X^\tau} \to E|_{X^\tau}$.

  Next, we extend our section further from $X^\tau \cup A$ to~$X$.  It
  suffices to extend an injective $\Z/2$-equivariant vector bundle map
  from the boundary of any cell in $(X,A \cup X^\tau)$ to the whole
  cell: if we can do this, we may build the required section by
  induction over the skeleta.  Since we work relative to~$X^\tau$,
  only free cells $\mathbb{D}^j\times \Z/2$ occur.  An injective
  \emph{$\Z/2$-equivariant} vector bundle map on
  $\mathbb{D}^j\times \Z/2$ is equivalent to an injective vector
  bundle map on one of the pieces~$\mathbb{D}^j$.  Here our problem
  becomes equivalent to extending a $\C$\nb-vector bundle isomorphism
  $E|_{\partial \mathbb{D}^j} \cong E_0 \oplus (\partial \mathbb{D}^j
  \times \C^2)$ for some $\C$\nb-vector bundle~$E_0$
  over~$\partial \mathbb{D}^j$ to a $\C$\nb-vector bundle isomorphism
  $E|_{\mathbb{D}^j} \cong \tilde{E}_0 \oplus (\mathbb{D}^j \times
  \C^2)$ for some $\C$\nb-vector bundle~$\tilde{E}_0$
  over~$\mathbb{D}^j$.  This amounts to applying
  Proposition~\ref{Husemoller} for $\field=\C$ twice and is possible if
  $j \le 2(k-1) -1 = 2 k -3$.  This is indeed the case because
  $j \le d_1 \le 2 k_0+1\le 2 k-3$.
\end{proof}

\begin{theorem}
  \label{the:main_quaternion_relative_2}
  Let \(d_1,d_0,k\in\N\).  Let \((X,A)\) be a relative
  \(\Z/2\)-CW-complex.  Assume that the free cells in~\((X,A)\) have
  at most dimension~\(d_1\) and that the trivial cells have at most
  dimension~\(d_0\).  Let
  \[
    k_1 \defeq
    \max \left\{
      \left\lceil \frac{d_0-2}{2} \right\rceil,
      \left\lceil \frac{d_1}{2} \right\rceil
    \right\}.
  \]
  Let \(E_1\) and~\(E_2\) be two ``quaternionic'' vector bundles
  over~\(X\) of rank \(k\ge k_1\).  An isomorphism
  $E_1|_A \cong E_2|_A$ that extends to a stable isomorphism between
  \(E_1\) and~\(E_2\) on~$X$ extends to an isomorphism
  \(E_1\cong E_2\).
\end{theorem}

\begin{proof}
  This follows from Theorem~\ref{the:main_quaternion_relative_1} in
  exactly the same way as Theorem~\ref{the:main_real_relative_2}
  follows from Theorem~\ref{the:main_real_relative_1}.  First, we use
  that any ``quaternionic'' vector bundle over a finite-dimensional
  $\Z/2$-CW-complex is a direct summand in a trivial ``quaternionic''
  bundle.  The relevant quantities for $d_1,d_0,k$ on $X\times I$ are
  now $d_1+1$, $d_0+1$ and $k+2$ because the smallest trivial
  ``quaternionic'' bundle has rank~$2$.  So the estimate about the
  rank in Theorem~\ref{the:main_quaternion_relative_1} for the
  extension problem on $X\times I$ is equivalent to the assumption
  made in this theorem.
\end{proof}

\section{Conjugacy of projections}
\label{sec:conjugacy}

Our physical motivation was about projections in the observable
algebra being homotopic.  In this very short section, we briefly
comment on the link between this original problem and our results on
vector bundles.

Let $E$ be a ``real'' or ``quaternionic'' vector bundle
over~$(X,\tau)$.  Then~$E$ is a direct summand in a trivial bundle.
Equivalently, there is another ``real'' or ``quaternionic'' vector
bundle~$E^\bot$ over~$X$ so that $E\oplus E^\bot \cong \theta^k_X$ for
some $k\in\N$, with $k\in 2\N$ in the ``quaternionic'' case.  The
projection onto~$E$ is an endomorphism of the trivial
bundle~$\theta^k_X$.  In the ``real'' case, the endomorphism ring
of~$\theta^k_X$ is the ring of functions $X\to \Mat_k(\C)$ that
satisfy $\conj{f(x)} = f(\tau(x))$.  In the ``quaternionic'' case, let
$\Theta_0 = \bigl(
\begin{smallmatrix}
  0&-1\\1&0
\end{smallmatrix}\bigr) \in \Mat_2(\R)$ and let $\Theta_0^{(k)} \in
\Mat_{2 k}(\R)$ be the block diagonal sum of $k$~copies of~$\Theta_0$.
Then the endomorphism ring of~$\theta^k_X$ is the ring of functions
$X\to \Mat_k(\C)$ that satisfy
\(\Theta_0^{(k)} \conj{f(x)} \Theta_0^{(k)} = f(\tau(x))\).  Taking
the pointwise adjoint of a matrix-valued function makes this
endomorphism ring into a unital $\Cst$\nb-algebra~$A$.  For $X=\T^d$,
this is the observable $\Cst$\nb-algebra for a translation-invariant
physical system in dimension~$d$ with a time-reversal symmetry of
square $+1$ in the ``real'' case and of square $-1$ in the
``quaternionic'' case.  Homotopy classes of projections in~$A$ are in
bijection with homotopy classes of invertible, self-adjoint elements
of~$A$, which are the possible Hamiltonians for insulators when the
observable algebra is~$A$.  So the physical question is to classify
the projections in~$A$ up to homotopy.

Each projection in~$A$ generates a direct sum decomposition of the
trivial bundle~$\theta^k_X$ as $E \oplus E^\bot$, where $E$
and~$E^\bot$ are the image bundles of $p$ and $1-p$, respectively.
The following result is well known.

\begin{lemma}
  Let $p$ and~$q$ be two such projections and let
  $\theta^k_X =E \oplus E^\bot$ and $\theta^k_X = F \oplus F^\bot$ be
  the resulting direct sum decompositions.  There is an invertible
  element $v\in A$ with $v p v^{-1} = q$ if and only if $E \cong F$
  and $E^\bot \cong F^\bot$ as ``real'' or ``quaternionic'' vector
  bundles.
\end{lemma}

\begin{proof}
  If $E \cong F$ and $E^\bot \cong F^\bot$, then the two isomorphisms
  together produce an automorphism
  $\theta^k_X =E \oplus E^\bot \cong F \oplus F^\bot = \theta^k_X$.
  This is simply an invertible element $v\in A$, and it satisfies
  $v p v^{-1} = q$.  Conversely, such an invertible element defines an
  automorphism of~$\theta^k_X$ that restricts to isomorphisms
  $E \cong F$ and $E^\bot \cong F^\bot$.
\end{proof}

We call $p$ and~$q$ \emph{conjugate} if there is an invertible element
$v\in A$ with $v p v^{-1} = q$.  This is well known to be equivalent
to the existence of a unitary $v\in A$ with $v p v^{-1} = q$.  It is
also well known that homotopic projections are conjugate.  The
converse is only known up to stabilisation, however.  The issue is
whether the invertible element implementing the conjugacy is homotopic
to the unit in~$A$.

The projections $p$ and~$q$ are stably conjugate, meaning that there
is a projection $r$ in another matrix algebra so that $p\oplus r$ and
$q\oplus r$ are conjugate, if and only if $E$ and~$F$ are stably
isomorphic and $E^\bot$ and~$F^\bot$ are stably isomorphic.  So
Proposition~\ref{pro:torus_dim_four} says that if $X=\T^d$ with
$d\le 4$, then stable conjugacy and conjugacy are equivalent for
projections in~$A$.  It is impossible, however, to prove a result that
relates stable homotopy and homotopy by working only with vector
bundles.

\end{document}